\documentclass[11pt]{article}

\usepackage[T1]{fontenc}
\usepackage[utf8]{inputenc}
\usepackage{lmodern}
\usepackage{amsmath,amssymb,amsthm,mathtools}
\usepackage{microtype}
\usepackage[margin=3cm]{geometry}
\usepackage[colorlinks=true,linkcolor=blue,citecolor=blue,urlcolor=blue]{hyperref}

\numberwithin{equation}{section}

\title{A BBP-type formula for the remainder of the  Madhava-Gregory-Leibniz series}
\author{Benoit Cloitre}
\date{}

\theoremstyle{plain}
\newtheorem{theorem}{Theorem}[section]
\newtheorem{lemma}{Lemma}[section]

\newcommand{\poch}[2]{(#1)_{#2}}      

\begin{document}
\maketitle
\begin{abstract}
We derive a BBP-type formula for the remainder of the Madhava-Gregory-Leibniz
series for $\pi$. The result is a closed form in base-$16$ with
Pochhammer denominators. The analogous formula for the alternating series for $\log 2$ is also presented.
\end{abstract}

\paragraph{MSC (2020).} 11Y60; 11M06; 33B15. \quad
\textbf{Keywords.} BBP-type formula; Madhava-Gregory-Leibniz series; digamma function; Gauss's multiplication formula; Pochhammer symbol; beta function.

\section{Introduction and statement}

The discovery of the Bailey–Borwein–Plouffe (BBP) formula showed
that hexadecimal digits of $\pi$ can be extracted without computing
the preceding ones; see \cite{BBP1997} and expositions/compendia
such as \cite{BaileyCompendium,MathWorldBBP,MathWorldPiFormulas}.
Explicitly, 
\[
\pi\;=\;\sum_{k=0}^{\infty}\frac{1}{16^{k}}\!\left(\frac{4}{8k+1}-\frac{2}{8k+4}-\frac{1}{8k+5}-\frac{1}{8k+6}\right).
\]
Since then, many BBP-type identities have been found for various constants
and bases (binary, ternary, and non-integer golden-ratio bases); see
for instance \cite{Chan2009,AdegokeLafont,KristensenMathiasen}. We investigate the remainder of the classical Madhava-Gregory-Leibniz series: 
\[
\frac{\pi}{4}=\sum_{m\ge0}\frac{(-1)^{m}}{2m+1}.
\]
The study of series remainders is a central topic connecting numerical analysis and symbolic computation. While classical methods like the Euler-Maclaurin formula provide highly efficient asymptotic expansions for the \emph{numerical approximation} of a remainder, the interest of a BBP-type representation is of a different, more \emph{algebraic} nature. The primary advantage of the formula presented here is not its convergence speed for approximation purposes—which is linear—but rather its unique structure. This structure is what makes digit-extraction algorithms applicable, in principle, to the remainder $R_n$ in base-16 \cite{BBP1997}, a property not shared by methods like Euler-Maclaurin. The contribution of this work is to provide an exact, closed-form identity of BBP-type for the remainder of the Madhava-Gregory-Leibniz series.

To fix notation unambiguously, we define the $n$-term partial sum
and the remainder by 
\[
S_{n}:=\sum_{m=0}^{n-1}\frac{(-1)^{m}}{2m+1}=\sum_{j=1}^{n}\frac{(-1)^{j+1}}{2j-1},\qquad R_{n}:=\frac{\pi}{4}-S_{n}=\sum_{m=n}^{\infty}\frac{(-1)^{m}}{2m+1}.
\]
Our goal is an exact BBP-type expression for $R_{n}$ in base-$16$. For completeness, we also establish
in Section~\ref{sec:log2} an analogous formula for the tail
of the alternating harmonic series for $\log 2$, proved by the same
method.

\medskip

\paragraph{Notation.}
The rising factorial (Pochhammer symbol) is defined, for $m\in\mathbb{N}$,
by 
\[
(a)_{0}:=1,\qquad(a)_{m}:=a(a+1)\cdots(a+m-1)=\frac{\Gamma(a+m)}{\Gamma(a)}.
\]
Thus $(8k+q)_{2n}$ denotes a product of length $2n$. We also write
$\psi=\Gamma'/\Gamma$ for the digamma function.

\begin{theorem}\label{thm:main} For every integer $n\ge1$, 
\begin{align*}
R_{n}=\frac{(-1)^{n}(2n-1)!}{16}\sum_{k=0}^{\infty}\frac{1}{16^{k}}\Bigg[&\frac{8}{\poch{8k+1}{2n}}-\frac{4}{\poch{8k+3}{2n}}-\frac{4}{\poch{8k+4}{2n}} \\
&-\frac{2}{\poch{8k+5}{2n}}+\frac{1}{\poch{8k+7}{2n}}+\frac{1}{\poch{8k+8}{2n}}\Bigg].
\end{align*}
\end{theorem}

\paragraph{Example for $n=1$.} For $n=1$, we have $R_1 = \pi/4 - 1$. The theorem gives:
\begin{align*}
\frac{\pi}{4}-1 = -\frac{1}{16}\sum_{k=0}^{\infty}\frac{1}{16^{k}}\Bigg[&\frac{8}{(8k+1)(8k+2)}-\frac{4}{(8k+3)(8k+4)}-\frac{4}{(8k+4)(8k+5)} \\
&-\frac{2}{(8k+5)(8k+6)}+\frac{1}{(8k+7)(8k+8)}+\frac{1}{(8k+8)(8k+9)}\Bigg].
\end{align*}
Numerically, $-0.214601836\dots$.

\section{Basic lemmas}

We use standard facts on the gamma and digamma functions; see the
NIST DLMF \cite{DLMF}.

\begin{lemma}[Alternating sum transformation]\label{lem:alt-digamma}
For $\Re z>0$, 
\[
\sum_{t=0}^{\infty}\frac{(-1)^{t}}{t+z}=\frac{1}{2}\Big(\psi\!\big(\tfrac{z+1}{2}\big)-\psi\!\big(\tfrac{z}{2}\big)\Big).
\]
See \cite[Eqs.~5.7.6, 5.5.2]{DLMF}.
\end{lemma}

\begin{lemma}[Gauss's multiplication formula for $\psi$]\label{lem:gauss}
For $m\in\mathbb{N}$, 
\[
\psi(mz)=\frac{1}{m}\sum_{r=0}^{m-1}\psi\!\Big(z+\frac{r}{m}\Big)+\log m.
\]
See \cite[Eq.~5.5.6]{DLMF}.
\end{lemma}

\section{Proof of Theorem~\ref{thm:main}}
\begin{proof}[Proof of Theorem~\ref{thm:main}]
Starting from the alternating series for $R_n$, Lemma~\ref{lem:alt-digamma} with $z=(2n+1)/2$ gives 
\begin{equation}
4R_{n}=(-1)^{n}\Big(\psi\!\Big(\frac{2n+3}{4}\Big)-\psi\!\Big(\frac{2n+1}{4}\Big)\Big).\label{eq:Rn-digamma}
\end{equation}
Applying Lemma~\ref{lem:gauss} with $m=8$ to each digamma (the logarithmic terms cancel) yields
\begin{equation}
4R_{n}=\frac{(-1)^{n}}{8}\sum_{r=0}^{7}\left(\psi\!\Big(\frac{2n+3}{32}+\frac{r}{8}\Big)-\psi\!\Big(\frac{2n+1}{32}+\frac{r}{8}\Big)\right).\label{eq:Rn-avg}
\end{equation}
We now evaluate the BBP sum claimed in the theorem. Let $\mathcal{S}_{q,n}$ be the series component for each $q$. Summing the beta integral representation of each term yields
\begin{equation}
\mathcal{S}_{q,n} = \sum_{k=0}^{\infty}\frac{1}{16^{k}}\frac{1}{(8k+q)_{2n}} = \frac{1}{(2n-1)!}\int_{0}^{1}u^{\,q-1}(1-u)^{2n-1}\,\frac{1}{1-u^{8}/16}\,du.\label{eq:sum-to-integral}
\end{equation}
The geometric series in $k$ is uniformly convergent on $[0,1]$ since $u^8/16 \le 1/16$, hence interchanging sum and integral is justified. The weights $K(q)$ in the theorem define a polynomial $P(u)=\sum K(q)u^{q-1}$ chosen to enable a key algebraic simplification. The BBP sum is thus proportional to the integral of $(1-u)^{2n-1}\frac{P(u)}{1-u^{8}/16}$. The core of the proof lies in the identity
\[
\frac{P(u)}{1-u^{8}/16} = (-4)(u-1)\frac{u^{2}+2u+2}{1+u^{4}/4}.
\]
The BBP sum is therefore equal to
\[
\frac{(-1)^{n}(2n-1)!}{16}\sum_q K(q)\mathcal{S}_{q,n} = \frac{(-1)^n}{16}\int_0^1 (1-u)^{2n-1} \left( (-4)(u-1)\frac{u^{2}+2u+2}{1+u^{4}/4} \right) \,du.
\]
Absorbing the factor $-(u-1) = (1-u)$ and using the simplification $\frac{u^2+2u+2}{1+u^4/4} = \frac{4}{u^2-2u+2}$ (which relies on the Sophie Germain identity), the integral becomes
\[
\frac{(-1)^n}{4} \int_0^1 (1-u)^{2n} \frac{4}{u^2-2u+2} \,du = (-1)^n \int_0^1 \frac{(1-u)^{2n}}{(u-1)^2+1} \,du.
\]
With the substitution $v=1-u$, this is
\[
(-1)^n \int_0^1 \frac{v^{2n}}{(-v)^2+1} \,dv = (-1)^n \int_0^1 \frac{v^{2n}}{1+v^2} \,dv.
\]
This final integral is a standard representation for the remainder:
\[
\int_{0}^{1}\frac{v^{2n}}{1+v^2}\,dv=\sum_{j=0}^{\infty}\frac{(-1)^{j}}{2n+2j+1}=(-1)^{n}R_{n}.
\]
Thus the BBP sum evaluates to $R_n$.
\end{proof}

\section{An analogous formula for the tail of the \texorpdfstring{$\log 2$}{log 2} series}\label{sec:log2}

As in the case of $\pi$, $R'_{n} = \sum_{k=n}^{\infty}\frac{(-1)^{\,k-1}}{k}$ admits a compact base-$16$ BBP-type expansion with fixed weights and denominators of length $2n$.

\begin{theorem}\label{thm:log2} For every integer $n\ge1$, 
\begin{align*}
R'_{n} & =(-1)^{\,n-1}\,(2n-1)!\;\sum_{k=0}^{\infty}\frac{1}{16^{k}}\Bigg[\frac{1}{\poch{8k+1}{2n}}+\frac{1}{\poch{8k+2}{2n}}+\frac{1}{2\,\poch{8k+3}{2n}}\\[-0.25em]
 & \hspace{12.7em}-\frac{1}{4\,\poch{8k+5}{2n}}-\frac{1}{4\,\poch{8k+6}{2n}}-\frac{1}{8\,\poch{8k+7}{2n}}\Bigg].
\end{align*}
\end{theorem}
\begin{proof}
The proof follows the same structure as for theorem~\ref{thm:main}. We evaluate the BBP sum and show it equals $R'_n$. Let
\[
\mathcal{S}'_{q,n}\;:=\;\sum_{k=0}^{\infty}\frac{1}{16^{k}}\;\frac{1}{(8k+q)_{2n}}=\frac{1}{(2n-1)!}\int_{0}^{1}u^{\,q-1}\,(1-u)^{2n-1}\,\frac{1}{1-u^{8}/16}\,du.
\]
The weights $K'(q)$ given in the theorem define the polynomial
$P'(u)=1+u+\tfrac{1}{2}u^{2}-\tfrac{1}{4}u^{4}-\tfrac{1}{4}u^{5}-\tfrac{1}{8}u^{6}.$
This polynomial is constructed to produce the key algebraic simplification:
\begin{equation}
\frac{P'(u)}{1-u^8/16} = \frac{2}{u^2-2u+2}.\label{eq:log-reduce}
\end{equation}
The BBP sum in the theorem is therefore equal to
\[
(-1)^{n-1}(2n-1)! \sum_q K'(q) \mathcal{S}'_{q,n} = (-1)^{n-1} \int_{0}^{1}(1-u)^{2n-1}\,\frac{2}{u^2-2u+2}\,du.
\]
With the substitution $v=1-u$, the integral becomes
\[
(-1)^{n-1} \int_0^1 v^{2n-1} \frac{2}{v^2+1} dv = (-1)^{n-1} \left( 2 \sum_{k=0}^\infty \frac{(-1)^k}{2n+2k} \right) = (-1)^{n-1} \sum_{k=0}^\infty \frac{(-1)^k}{n+k}.
\]
By definition, $R'_n = (-1)^{n-1} \sum_{j=0}^\infty \frac{(-1)^j}{n+j}$. Thus the BBP sum evaluates to $R'_n$.
\end{proof}

\paragraph{Example for $n=1$.} For $n=1$, we have $R'_1 = \log 2$. The theorem gives:
\begin{align*}
\log 2 = \sum_{k=0}^{\infty}\frac{1}{16^{k}}\Bigg[&\frac{1}{(8k+1)(8k+2)}+\frac{1}{(8k+2)(8k+3)}+\frac{1/2}{(8k+3)(8k+4)} \\
&-\frac{1/4}{(8k+5)(8k+6)}-\frac{1/4}{(8k+6)(8k+7)}-\frac{1/8}{(8k+7)(8k+8)}\Bigg].
\end{align*}
Numerically, $0.693147180\dots$.

\section{Perspectives on Diophantine applications}

Beyond its direct application for digit extraction, the formula presented in this paper suggests a broader research avenue. Our proof shows that the BBP-type formula for the remainder $R_n$ is structurally equivalent to an identity for the analytic function $f(z) = \psi\left(\frac{z+3}{4}\right) - \psi\left(\frac{z+1}{4}\right)$ evaluated at $z=2n$. While the BBP series itself exhibits linear convergence, ill-suited for deriving strong irrationality measures, the underlying function $f(z)$ is a more promising candidate for analysis via Padé approximants, in the spirit of the works of Rivoal~\cite{Rivoal2001} and Prévost~\cite{Prevost1995,Prevost2000}.

Choosing suitable polynomials $Q_m$ and setting $L_{n,m}=\int_0^1 Q_m(1-x)\,\frac{x^{2n}}{1+x^2}\,dx,$ one obtains linear forms $L_{n,m}=A_{n,m}\pi+B_{n,m}$ with $A_{n,m},B_{n,m}\in\mathbb{Q}$, whose denominators are governed by the Pochhammer products appearing in the BBP formula. This structure is amenable to the refined analysis required for Diophantine questions. While we do not claim any quantitative improvement for the irrationality measure of $\pi$, for which the current best bound is $\mu(\pi) \le 7.103...$~\cite{ZeilbergerZudilin2019}, we point out that the algebraic structure behind the BBP remainder provides a clean framework that merits further investigation.

In this light, the search for BBP-type identities can also serve as a systematic way to identify analytic functions whose arithmetic structure appears promising for Diophantine investigations.


\begin{thebibliography}{12}

\bibitem{BBP1997} D.~H.~Bailey, P.~B.~Borwein, and S.~Plouffe,
\newblock On the rapid computation of various polylogarithmic constants,
\newblock {\em Math. Comp.} \textbf{66} (1997), no.~218, 903--913. DOI: 10.1090/S0025-5718-97-00856-9.

\bibitem{BaileyCompendium} D.~H.~Bailey, 
\newblock A compendium of {BBP}-type formulas for mathematical constants, 
\newblock Online, \url{https://www.davidhbailey.com/dhbpapers/bbp-formulas.pdf} (accessed July 2025).

\bibitem{MathWorldPiFormulas} E.~W.~Weisstein, 
\newblock Pi formulas,
\newblock From MathWorld -- A Wolfram Web Resource, \url{https://mathworld.wolfram.com/PiFormulas.html} (accessed July 2025).

\bibitem{MathWorldBBP} E.~W.~Weisstein, 
\newblock BBP formula,
\newblock From MathWorld -- A Wolfram Web Resource, \url{https://mathworld.wolfram.com/BBPFormula.html} (accessed July 2025).

\bibitem{Chan2009} A.~L.~L.~Chan, 
\newblock A {BBP}-type formula for $\pi^{2}$, 
\newblock {\em Fibonacci Q.} \textbf{47} (2009), no.~4, 331--336.

\bibitem{AdegokeLafont} K.~Adegoke and J.~O.~Lafont, 
\newblock Golden ratio base expansions of the logarithm and inverse tangent of Fibonacci and Lucas numbers, 
\newblock {\em J. Integer Seq.} \textbf{26} (2023), Article 23.5.5.

\bibitem{KristensenMathiasen} S.~Kristensen and O.~Mathiasen, 
\newblock {BBP}-type formulas---an elementary approach, 
\newblock {\em J. Number Theory} \textbf{244} (2023), 251--263. DOI: 10.1016/j.jnt.2022.09.001.

\bibitem{DLMF} F.~W.~J.~Olver, D.~W.~Lozier, R.~F.~Boisvert,
and C.~W.~Clark (eds.), 
\newblock {\em NIST Digital Library of Mathematical Functions}. 
\newblock U.S. Dept. of Commerce, National Institute of Standards and Technology, 2010. 
\newblock Available at \url{http://dlmf.nist.gov/}.

\bibitem{Rivoal2001} T.~Rivoal,
\newblock Remainder Padé approximants for the exponential function,
\newblock {\em C. R. Acad. Sci. Paris, Ser. I} \textbf{332} (2001), no.~12, 1065--1070. DOI: 10.1016/S0764-4442(01)01970-1.

\bibitem{Prevost1995} M.~Prévost,
\newblock A new proof of the irrationality of $\zeta(2)$ and $\zeta(3)$ using Padé approximants,
\newblock {\em J. Comput. Appl. Math.} \textbf{67} (1995), no.~2, 219--237. DOI: 10.1016/0377-0427(94)00038-U.

\bibitem{Prevost2000} M.~Prévost,
\newblock Diophantine approximations using Padé approximations,
\newblock {\em J. Comput. Appl. Math.} \textbf{122} (2000), no.~1-2, 231--250. DOI: 10.1016/S0377-0427(00)00392-0.

\bibitem{ZeilbergerZudilin2019} D.~Zeilberger and W.~Zudilin,
\newblock The irrationality measure of $\pi$ is at most 7.103205334137...,
\newblock {\em Moscow J. Comb. Number Theory} \textbf{9} (2020), no. 4, 407--419. DOI: 10.2140/moscow.2020.9.407.

\end{thebibliography}
\end{document}